\documentclass{amsart}
\usepackage{graphicx, amsmath, amssymb}
\usepackage{color}

\makeatletter \oddsidemargin.9375in \evensidemargin \oddsidemargin
\def\Corresponding author{$^{*}$\protect\footnotetext{$^{*}$ C\lowercase{orresponding author.}}}
\def\authorsaddresses#1{\dedicatory{#1}}

\newtheorem{thm}{Theorem}[section]
\theoremstyle {definition}
\newtheorem{cor}[thm]{Corollary}
\newtheorem{prop}[thm]{Proposition}

\newtheorem{lem}[thm]{Lemma}
\newtheorem{eg}[thm]{Example}

\newtheorem{rem}[thm]{Remark}

\numberwithin{equation}{section}

\begin{document}
\setcounter{page}{1}

\title[The Zariski topology-graph of modules]{The Zariski topology-graph of modules over commutative rings II}

\author[H. Ansari-Toroghy and S. Habibi]{H. Ansari-Toroghy$^1$ and S. Habibi$^2$}

\authorsaddresses{$^1$ Department of pure Mathematics,\\ Faculty of mathematical Sciences,\\ University of Guilan,
P. O. Box 41335-19141, Rasht, Iran.\\
e-mail: ansari@guilan.ac.ir\\
\vspace{0.5cm} $^2$ School of Mathematics, Institute for Research
in Fundamental Sciences (IPM), P.O. Box: 19395-5746, Tehran,
Iran.\\ Department of pure Mathematics, Faculty of mathematical
Sciences, University of Guilan, P. O. Box 41335-19141, Rasht,
Iran.
\\
e-mail: habibishk@gmail.com} 
\keywords{Rings and modules; Zariski topology; graph; chromatic
and clique number.\\This research was in part supported by a grant
from IPM (No. 96130028)}
\begin{abstract}
Let $M$ be a module over a commutative ring $R$. In this paper, we
continue our study about the Zariski topology-graph $G(\tau_T)$
which was introduced in (The Zariski topology-graph of modules
over commutative rings, Comm. Algebra., 42 (2014), 3283--3296).
For a non-empty subset $T$ of $Spec(M)$, we obtain useful
characterizations for those modules $M$ for which $G(\tau_T)$ is a
bipartite graph. Also, we prove that if $G(\tau_T)$ is a tree,
then $G(\tau_T)$ is a star graph. Moreover, we study coloring of
Zariski topology-graphs and investigate the interplay between
$\chi(G(\tau_T))$ and $\omega(G(\tau_T))$.
\end{abstract}
\maketitle
\section{Introduction}
 Throughout this paper $R$ is a commutative ring with a non-zero
identity and $M$ is a unital $R$-module. By $N\leq M$ (resp. $N<
M$) we mean that $N$ is a submodule (resp. proper submodule) of
$M$.

Define $(N:_{R}M)$ or simply $(N:M)=\{r\in R|$ $rM\subseteq N\}$
for any $N\leq M$. We denote $((0):M)$ by $Ann_{R}(M)$ or simply
$Ann(M)$. $M$ is said to be faithful if $Ann(M)=(0)$.

Let $N, K\leq M$. Then the product of $N$ and $K$, denoted by
$NK$, is defined by $(N:M)(K:M)M$ (see \cite{af07}).

A prime submodule of $M$ is a submodule $P\neq M$ such that
whenever $re\in P$ for some
  $r\in R$ and $e \in M$, we have $r\in (P:M)$ or $e\in P$ \cite{lu84}.

The prime spectrum of $M$ is the set of all prime submodules of
$M$ and denoted by $Spec(M)$.

There are many papers on assigning graphs to rings or modules
(see, for example, \cite{al99, ah161, ah162, b88}). In
\cite{ah14}, the present authors introduced and studied the graph
$G(\tau_T)$ (resp. $AG(M)$), called the \textit{Zariski
topology-graph} (resp. \textit {the annihilating-submodule
graph}), where $T$ is a non-empty subset of $Spec(M)$.

$AG(M)$ is an undirected graph with vertices $V(AG(M))$= $\{N \leq
M |$ there exists $(0)\neq K<M$ with $NK=(0)$\}. In this graph,
distinct vertices $N,L \in V(AG(M))$ are adjacent if and only if
$NL=(0)$. Let $AG(M)^{*}$ be the subgraph of $AG(M)$ with vertices
$V(AG(M)^{*})=\{ N<M$ with $(N:M)\neq Ann(M)|$ there exists a
submodule $K<M$ with $(K:M)\neq Ann(M)$ and $NK=(0)\}$. By
\cite[Theorem 3.4]{ah14}, one conclude that $AG(M)^{*}$ is a
connected subgraph.

$G(\tau_T)$ is an undirected graph with vertices
 $V(G(\tau_T))$= $\{N < M|$ there exists $K < M$ such that $V(N)\cup V(K)=T$
  and $V(N),V(K)\neq T\}$ and distinct
   vertices $N$ and $L$ are adjacent if and only if $V(N)\cup V(L)=T$
    (see \cite[Definition 2.3]{ah14}).

The \textit{Zariski topology } on $X=Spec(M)$ is the topology
$\tau_M $ described by taking the set $\ Z(M)= \{ V(N)|$ $N$ is a
submodule of $ M \}$ as the set of closed sets of $Spec_{R}(M)$,
where $V(N)=\{P \in X|$ $(P:M)\supseteq (N:M) \}$ \cite{lu99}.

 If $Spec(M)\neq\emptyset$, the mapping $\psi
:Spec(M)\rightarrow Spec(R/Ann(M))$ such that $\psi
(P)=(P:M)/Ann(M)$ for every  $P \in Spec(M)$, is called the
\textit{natural map} of $Spec(M)$ \cite{lu99}.

A topological space $X$ is irreducible if for any decomposition
$X= X_1 \cup X_2$ with closed subsets $X_i$ of $X$ with $i= 1, 2$,
we have $X= X_1$ or $X= X_2$

The prime radical $\sqrt{N}$ is defined to be the intersection of
all prime submodules of $M$ containing $N$, and in case $N$ is not
contained in any prime submodule, $\sqrt{N}$ is defined to be $M$
\cite{lu84}.

We recall that $N< M$ is said to be a semiprime submodule of $M$
if for every ideal $I$ of $R$ and every submodule $K$ of $M$ with
$I^{2}K\subseteq N$ implies that $IK\subseteq N$. Further $M$ is
called a semiprime module if $(0)\subseteq M$ is a semiprime
submodule. Every intersection of prime submodules is a semiprime
submodule (see \cite{tv08}).

The notations $Nil(R)$, $Min(M)$, and $Min(T)$ will denote the set
of all nilpotent elements of $R$ and the set of all minimal prime
submodules of $M$, and the set of minimal members of $T$,
respectively.

A clique of a graph is a complete subgraph and the supremum of the
sizes of cliques in $G$, denoted by $\omega(G)$, is called the
clique number of $G$. Let $\chi(G)$ denote the chromatic number of
the graph $G$, that is, the minimal number of colors needed to
color the vertices of $G$ so that no two adjacent vertices have
the same color. Obviously $\chi(G)\geq \omega(G)$.

In this article, we continue our studying about $G(\tau_T)$ and
$AG(M)$ and we try to relate the combinatorial properties of the
above mentioned graphs to the algebraic properties of $M$.

In section 2 of this paper, we state some properties related to
the Zariski topology-graph that are basic or needed in the later
sections. In section 3, we study the bipartite Zariski
topology-graphs of modules over commutative rings (see Proposition
\ref{p3.1}). Also, we prove that if $G(\tau_T)$ is a tree, then
$G(\tau_T)$ is a star graph (see Theorem \ref{t3.4}). In section
4, we study coloring of the Zariski topology-graph of modules and
investigate the interplay between $\chi(G(\tau_T))$ and
$\omega(G(\tau_T))$. We show that under condition over minimal
submodules of $M/(\cap_{P\in T}P:M)M $, we have
$\omega(G(\tau_T))= \chi(G(\tau_T))$ (see Theorem \ref{t4.1}).
Moreover, we investigate some relations between the existence of
cycles in the Zariski topology-graph of a cyclic module and the
number of its minimal members of $T$ (see Proposition \ref{p4.8}).

Let us introduce some graphical notions and denotations that are
used in what follows: A graph $G$ is an ordered triple $(V(G),
E(G), \psi_G )$ consisting of a nonempty set of vertices,
 $V(G)$, a set $E(G)$ of edges, and an incident function $\psi_G$ that associates an
 unordered pair
 of distinct vertices with each edge. The edge $e$ joins $x$ and $y$ if $\psi_G(e)=\{x, y\}$, and we
 say $x$ and $y$ are adjacent. A path in graph $G$ is a finite sequence of vertices $\{x_0,
x_1,\ldots ,x_n\}$, where $x_{i-1}$ and $x_i$ are adjacent for
each $1\leq i\leq n$ and we denote $x_{i-1} - x_i$ for existing an
edge between
 $x_{i-1}$ and $x_i$.

A graph $H$ is a subgraph of $G$, if $V(H)\subseteq V(G)$,
$E(H)\subseteq E(G)$, and $\psi_H$ is the restriction of $\psi_G$
to $E(H)$. A bipartite graph is a graph whose vertices can be
divided into two disjoint sets $U$ and $V$ such that every edge
connects a vertex in $U$ to one in $V$; that is, $U$ and $V$ are
each independent sets and complete bipartite graph on $n$ and $m$
vertices, denoted by $K_{n, m}$, where $V$ and $U$ are of size $n$
and $m$, respectively, and $E(G)$ connects every vertex in $V$
with all vertices in $U$. Note that a graph $K_{1, m}$ is called a
star graph and the vertex in the singleton partition is called the
center of the graph. For some $U\subseteq V (G)$, we denote by
$N(U)$, the set of all vertices of $G\setminus U$ adjacent to at
least one vertex of $U$. For every vertex $v\in V(G)$, the size of
$N(v)$ is denoted by $deg(v)$. If all the vertices of $G$ have the
same degree $k$, then $G$ is called $k$-regular, or simply
regular. We denote by $C_{n}$ a cycle of order $n$. Let $G$ and
$G'$ be two graphs. A graph homomorphism from $G$ to $G'$ is a
mapping $\phi: V(G) \longrightarrow V(G')$ such that for every
edge $\{u, v\}$ of $G$, $\{\phi(u), \phi(v)\}$ is an edge of $G'$.
A retract of $G$ is a subgraph $H$ of $G$ such that there exists a
homomorphism $\phi: G \longrightarrow H$ such that $\phi(x)= x$,
for every vertex $x$ of $H$. The homomorphism $\phi$ is called the
retract (graph) homomorphism (see \cite{r05}).

Throughout the rest of this paper, we denote: $T$ is a non-empty
subset of $Spec(M)$, $Q:= (\cap_{P\in T}P:M)M$, $\bar{M}:=M/ Q$,
$\bar{N}:=N/Q$, $\bar{m}:=m+Q$, and $\bar{I}:=I/(Q:M)$, where $N$
is a submodule of $M$ containing $Q$, $m\in M$, and $I$ is an
ideal of $R$ containing $(Q:M)$.

\section{Auxiliary results}
In this section, we provide some properties related to the Zariski
topology-graph that are basic or needed in the sequel.

\begin{rem}\label{r2.1}
Let $N$ be a submodule of $M$. Set $V^*(N):= \{P \in Spec(M) |$
$P\supseteq N \}$. By \cite [Remark 2.2]{ah14}, For submodules $N$
and $K$ of $M$, we have
$$V(N)\cup V(K)=V(N\cap K)=V(NK)=V^*(NK).$$ By \cite [Remark 2.5]{ah14}, we have $T$ is a
 closed subset of $Spec(M)$ if and
only if $T=V(\cap_{P\in T}P)$ and $G(\tau_T)\neq \emptyset $ if
and only if $T=V(\cap_{P\in T}P)$ and $T$ is not irreducible. So
if $N$ and $K$ are adjacent in $G(\tau_T)$, then
 $V^*(NK)=V^*((\cap_{P\in T}P:M)M)$ and hence $\sqrt{NK}=\cap_{P\in T}P$.
  Therefore $\cap_{P\in T}P \subseteq \sqrt{(N:M)M}, \sqrt{(K:M)M}$.
\end{rem}

\begin{lem}\label{l2.2} (See \cite[Proposition 7.6]{af74}.)
Let $R_{1}, R_{2}, \ldots , R_{n}$ be non-zero ideals of $R$. Then
the following statements are equivalent:

\begin{itemize}
\item [(a)] $R= R_{1} \oplus \ldots \oplus R_{n}$; \item [(b)] As
an abelian group $R$ is the direct sum of $ R_{1}, \ldots ,
R_{n}$; \item [(c)] There exist pairwise orthogonal idempotents
$e_{1},\ldots, e_{n}$ with $1=e_{1}+ \ldots +e_{n}$, and
$R_{i}=Re_{i}$, $i=1, \ldots ,n$.
\end{itemize}
\end{lem}

\begin{prop}\label{p2.3} Suppose that $e$ is an idempotent element of
$R$. We have the following statements.

\begin {itemize}
\item [(a)] $R=R_{1}\oplus R_{2}$, where $R_{1}=eR$ and
$R_{2}=(1-e)R$. \item [(b)] $M=M_{1}\oplus M_{2}$, where
$M_{1}=eM$ and $M_{2}=(1-e)M$. \item [(c)] For every submodule $N$
of $M$, $N=N_{1}\oplus N_{2}$ such that $N_{1}$ is an
$R_{1}$-submodule $M_{1}$, $N_{2}$ is an $R_{2}$-submodule
$M_{2}$, and $(N:_{R}M)=(N_{1}:_{R_{1}}M_{1})\oplus
(N_{2}:_{R_{2}}M_{2})$.  \item [(d)] For submodules $N$ and $K$ of
$M$, $NK=N_{1}K_{1} \oplus N_{2}K_{2}$, $N\cap K=N_{1}\cap K_{1}
\oplus N_{2}\cap K_{2}$ such that $N=N_{1}\oplus N_{2}$ and
$K=K_{1}\oplus K_{2}$. \item[(e)] Prime submodules of $M$ are
$P\oplus M_{2}$ and $M_{1}\oplus Q$, where $P$ and $Q$ are prime
submodules of $M_{1}$ and $M_{2}$, respectively. \item [(f)] For
submodule $N$ of $M$, we have $\sqrt{N}= \sqrt{N_{1}\oplus N_{2}}=
\sqrt{N_{1}}\oplus \sqrt{N_{2}}$, where $N=N_{1}\oplus N_{2}$.
\end{itemize}

\end{prop}

\begin{proof}
This is clear.
\end{proof}

An ideal $I< R$ is said to be nil if $I$ consist of nilpotent
elements.

\begin{lem}\label{l2.4} (See \cite[Theorem 21.28]{l91}.)
Let $I$ be a nil ideal in $R$ and $u\in R$ be such that
 $u+I$ is an idempotent in $R/I$. Then there exists an idempotent
 $e$ in $uR$ such that $e-u\in I$.
\end{lem}

\begin{lem}\label{l2.5} (See \cite[Lemma 2.4]{ah161}.)
Let $N$ be a minimal submodule of $M$ and let $Ann(M)$ be a nil
ideal. Then we have $N^{2}=(0)$ or $N=eM$ for some idempotent
$e\in R$.
\end{lem}

We note that $M$ is said to be \textit{primeful} if either $M=(0)$
or $M\neq (0)$ and the natural map of $Spec(M)$ is surjective (see
\cite{lu07}).

\begin{prop}\label{p2.6} We have the following statements.
\begin{itemize}
\item [(a)] If $N,L$ are adjacent in $G(\tau_T)$, then
$\sqrt{(N:M)M}/ \cap_{P\in T}P$ and\\ $\sqrt{(L:M)M}/ \cap_{P\in
T}P$ are adjacent in $AG(M/ \cap_{P\in T}P)$. \item [(b)] If $M$
is a primeful module and $N,L$ are adjacent in $G(\tau_T)$, then
$\sqrt{N}/ \cap_{P\in T}P$ and $\sqrt{L}/ \cap_{P\in T}P$ are
adjacent in $AG(M/ \cap_{P\in T}P)$.
\end{itemize}
\end{prop}
\begin{proof}
(a) First we see easily that for any submodule $N$ of $M$,
$V(N)=V(\sqrt{(N:M)M})$. Suppose that $N$ and $L$ are adjacent in
$G(\tau_T)$ so that $V(N) \cup V(L)= T$. Then we have
$V^{*}(\sqrt{(N:M)M} \sqrt{(L:M)M})= T$. It follows that
$\sqrt{(N:M)M} \sqrt{(L:M)M}\subseteq \cap_{P\in T}P$ (see Remark
\ref{r2.1}). Now the claim follow by Remark \ref{r2.1}.

 (b) This is clear by \cite[Corollary 4.5]{ah14}.
\end{proof}

\begin{rem}\label{r2.11} The Proposition \ref{p2.6} (a) extends \cite
[Theorem 4.4]{ah14}.
\end{rem}

\begin{lem}\label{l2.7} Assume that $T$ is a closed subset of $Spec(M)$.
Then $AG(\bar{M})^{*}$ is isomorphic with a subgraph of
$G(\tau_T)$. In particular, $AG(M/ \cap_{P\in T}P)^{*}$ is
isomorphic with an induced subgraph of $G(\tau_T)$.
\end{lem}
\begin{proof} Let $\bar{N} \in V(AG(\bar{M})^{*})$. Then there exists a nonzero submodule $\bar{K}$
of $\bar{M}$ such that it is adjacent to $\bar{N}$ (if $N=K$, then
$(N:M)=(Q:M)$, a contradiction). So we have $NK\subseteq Q$. Hence
$V(NK)=T$.
 If $V(N)=T$, then $(N:M)=(Q :M)$, a contradiction. Hence
 $N$ is a vertex in $G(\tau_T)$ which is adjacent to $L$. To see
 the last assertion, let $N/\cap_{P\in T}P$ and $K/\cap_{P\in T}P$ be two vertices
 of $AG(M/ \cap_{P\in T}P)^{*}$. If $N$ and
 $K$ are adjacent in $G(\tau_T)$, then by Proposition \ref{p2.6}, $\sqrt{(N:M)M}/ \cap_{P\in T}P$
 and $\sqrt{(K:M)M}/ \cap_{P\in T}P$ are adjacent in $AG(M/ \cap_{P\in T}P)^{*}$. So
 $$\sqrt{(N:M)M}  \sqrt{(L:M)M} \subseteq \cap_{P\in T}P.$$ Since
 $$NK=((N:M)M:M)((K:M)M:M)M\subseteq \sqrt{(N:M)M} \sqrt{(L:M)M},$$
 we have $N/\cap_{P\in T}P$ and $K/\cap_{P\in T}P$ are adjacent in $AG(M/ \cap_{P\in T}P)^{*}$, as desired.
\end{proof}

\begin{lem}\label{l2.8}
If $\bar{M}$ is a faithful module, then $G(\tau_{Spec(M)})$ and
$AG(M)^{*}$ are the same.
\end{lem}
\begin{proof}
$\bar{M}$ is a faithful module so that $T=Spec(M)$. If
$G(\tau_{Spec(M)})\neq \emptyset$,
 then there exist non-trivial submodules $N$ and $K$ of $M$ which is adjacent
  in $G(\tau_{Spec(M)})$. Hence $V(NK)=Spec(M)$
 which implies that $NK=(0)$ so that $AG(M)^{*}\neq \emptyset$.
  By Lemma \ref{l2.7}, $AG(M)^{*}$
 is isomorphic with a subgraph of
$G(\tau_{Spec(M)})$. One
 can see that the vertex map $\phi: V(G(\tau_{Spec(M)})) \longrightarrow
V(AG(M)^{*})$, defined by $N\longrightarrow N$ is an isomorphism.
\end{proof}

Recall that $\Delta(G(\tau_T))$ is the maximum degree of
$G(\tau_T)$ and the length of an $R$-module $M$, is denoted by
$l_{R}(M)$.

\begin{lem}\label{l2.9}
Let every nontrivial submodule of $M$ be a vertex in $G(\tau_T)$.
If $\Delta(G(\tau_T))< \infty$, then $l_{R}(M)\leq
\Delta(G(\tau_T))+1$. Also, every non-trivial submodule of $M$ has
finitely many submodules.
\end{lem}
\begin{proof}
First we show that the descending chain of non-trivial submodules
$K_{1}\supsetneq K_{2}\supsetneq K_{3}\supsetneq \ldots$
terminates. Since $G(\tau_T)$ is connected, there exists a
submodule $N$ such that $V(N)\cup V(K_{1})=T$. Hence for each $i$,
$i\geq 1$, $V(N)\cup V(K_{i})=T$ and so $deg(N)=\infty$, a
contradiction. Next, let $N_{1}\subsetneq N_{2}\subsetneq
N_{3}\subsetneq \ldots$ be an ascending chain of non-trivial
submodules of $M$. Since $G(\tau_T)$ is connected, there exists a
submodule $K$ such that $V(K)\cup V(N_{\Delta+1})=T$, where
$\Delta= \Delta(G(\tau_T))$. Hence $V(K)\cup V(N_{i})=T$ for each
$1\leq i\leq \Delta+1$. Thus $deg(K)\geq \Delta+1$, a
contradiction. It follows that $l_{R}(M)\leq \Delta+1$. For the
proof of the last assertion, let $N$ be a non-trivial submodule of
$M$. Since $G(\tau_T)$ is connected, there exists a submodule $K$
such that $V(N)\cup V(K)=T$. Hence for every submodule $N'$ of
$N$, $V(N')\cup V(K)=T$. As $\Delta< \infty$, the number of
submodules of $N$ should be finite.
\end{proof}

\begin{thm}\label{t2.10} Let $\bar{M}$ be a multiplication module and $G(\tau_T)\neq \emptyset$.
Then $G(\tau_T)$ has acc (resp. dcc)
 on vertices if and only if $\bar{M}$ is a Noetherian (resp. an Artinian) module.
\end{thm}

\begin{proof} Suppose that $G(\tau_T)$ has acc (resp. dcc) on vertices.
By \cite [Remark 2.6]{ah14}, $\bar{M}$ is not a prime module and
hence there exists $r \in R$ and $\bar{m}\in \bar{M}$ such that
$r\bar{m}=\bar{0}$ but $\bar{m}\neq \bar{0}$ and $r \notin
Ann(\bar{M})$. Now $\overline{rM}\cong
\bar{M}/(\bar{0}:_{\bar{M}}r)$. Further, $\overline{rM}$ and
$(\bar{0}:_{\bar{M}}r)$ are vertices because
$(\bar{0}:_{\bar{M}}r)(\overline{rM})=((\bar{0}:_{\bar{M}}r):\bar{M})(\overline{rM}:\bar{M})\bar{M}\subseteq
\overline{rM}((\bar{0}:_{\bar{M}}r):\bar{M})\subseteq
r(\bar{0}:_{\bar{M}}r)=\bar{0}$. Then $\{\bar{N}|$ $\bar{N}\leq
\bar{M}, \bar{N}\subseteq \overline{rM}\}\cup \{\bar{N}:
\bar{N}\leq \bar{M}, \bar{N}\subseteq (\bar{0}:_{\bar{M}}r)\}
\subseteq V(G(\tau_T))$.
 It follows that the $R$-modules $\overline{rM}$ and $(\bar{0}:_{\bar{M}}r)$ have acc (resp. dcc) on submodules.
 Since $\overline{rM}\cong \bar{M}/(\bar{0}:_{\bar{M}}r)$, $\bar{M}$ has acc on submodules and the proof is completed.
\end{proof}

\section{Zariski topology-graph of modules}
First, in this section we give the more notation to be used
throughout the remainder of this article. Suppose that $e$ $(
e\neq 0, 1)$ is an idempotent element of $R$. Let $M_{1}:=eM,
M_{2}:=(1-e)M, T_{1}:=\{P_{1}\in Spec(M_{1})| P_{1}\oplus M_{2}\in
T\}$, $T_{2}:=\{P_{2}\in Spec(M_{2})| M_{1}\oplus P_{2}\in T\}$,
$Q_{1}:=(\cap_{P_{1}\in T_{1}}P_{1}:M_{1})M_{1},
Q_{2}:=(\cap_{P_{2}\in T_{2}}P_{2}:M_{2})M_{2}$, $\bar{M_{1}}=
\overline{eM}=eM/Q_{1}$, and $\bar{M_{2}}=
\overline{(e-1)M}=(e-1)M/Q_{2}$. Consequently we have,
$Q=Q_{1}\oplus Q_{2}$, where $Q= (\cap_{P\in T}P:M)M$ and
$\bar{M}\cong \bar{M_{1}}\oplus \bar{M_{2}}$

We recall that a submodule $N$ of $M$ is a prime $R$-module if and
only if it is a prime $R/Ann(M)$-module (see \cite[Result
1.2]{ah14}).

\begin{prop}\label{p3.1} Suppose that $\bar{M}$ does not have a non-zero submodule
 $\overline{\cap_{P\in T}P}\neq \bar{N}$ with $V(N)=T$. Then the following statements hold.
\begin{itemize}
\item [(a)] If there exists a vertex of $G(\tau_T)$ which is
adjacent to every other vertex, then $\bar{M_{1}}$ is a simple
module and $\bar{M_{2}}$ is a prime module for some idempotent
element $e\in R$. \item [(b)] If $\bar{M_{1}}$ and $\bar{M_{2}}$
are prime modules for some idempotent element $e\in R$, then
$G(\tau_T)$ is a complete bipartite graph.
\end{itemize}
\end{prop}

\begin{proof}
$(a)$ Suppose that $N$ is adjacent to every other vertex of
$G(\tau_T)$. Since $V(N)=V((N:M)M)$, we have $N=(N:M)M$ and hence
$V(N)=V^{*}(N)$. Thus $N=\sqrt{N}$ because $V(N)=V(\sqrt{N})$. We
claim that $\bar{N}$ is a minimal submodule of $\bar{M}$. Let $Q
\subsetneq K \subsetneq N$. If $V(K)\neq T$, then $K$ is adjacent
to $N$ and hence $V(K)=T$, a contradiction. So $\bar{N}$ is a
minimal submodule of $\bar{M}$. We have $(\bar{N})^{2}\neq (0)$
because $V(N)\neq T$. Then Lemma \ref{l2.5}, implies that
$\bar{M}\cong \overline{eM}\oplus \overline{(e-1)M}$ for some
idempotent element $e$ of $R$. Without loss of generality we may
assume that $M_{1}\oplus Q_{2}$ is adjacent to every other vertex.
We claim that $\bar{M_{1}}$ is a simple module and $\bar{M_{2}}$
is a prime module. Let $Q_{1}\subsetneq K < M_{1}$. We have
$V(K\oplus Q_{2})\neq T$ because $Q_{1}\oplus Q_{2}\subsetneq
K\oplus Q_{2}$. Since $V(K\oplus Q_{2})\cup V(Q_{1}\oplus
M_{2})=T$, we have $K\oplus Q_{2}$ is a vertex and hence is
adjacent to $M_{1}\oplus Q_{2}$. Therefore $V(K\oplus Q_{2})\cup
V(M_{1}\oplus Q_{2})=V(K\oplus Q_{2})=T$, a contradiction. It
implies that $\bar{M_{1}}$ is a simple module. Now, we show that
$\bar{M_{2}}$ is a prime module. It is enough to show that is a
prime $R/(Q_{2}:M_{2})$-module. Otherwise, $\bar{I}\bar{K}=
(\bar{0})$, where $(Q_{2}:M_{2})\subsetneq I< R$ and
$Q_{2}\subsetneq K< M$. It follows that $V(M_{1}\oplus K) \cup
V(Q_{1}\oplus IM_{2})=V(Q_{1}\oplus K(IM_{2}))=T$ because
$K(IM_{2})\subseteq IK\subseteq Q_{2}$ and
$(Q_{2}:M_{2})^{2}M_{2}\subseteq K(IM_{2})$ (note that
$(Q_{2}:M_{2})\subseteq (K:M)$ and $(Q_{2}:M_{2})\subseteq I$).
Therefore $V(M_{1}\oplus K)\cup V(M_{1}\oplus
Q_{2})=T=V(M_{1}\oplus Q_{2})$, a contradiction (note that
$M_{1}\oplus K$ is properly containing $Q_{1}\oplus Q_{2}).$

$(b)$ Assume that $N_{1}\oplus N_{2}$ is adjacent to $K_{1}\oplus
K_{2}$. One can see that $\sqrt{N_{1}K_{1}}\oplus
\sqrt{N_{2}K_{2}}=\sqrt{Q_{1}}\oplus \sqrt{Q_{2}}$. It implies
that $\overline{(\sqrt{(K_{1}:M_{1})M_{1}}:M_{1})}$ $
\overline{\sqrt{( N_{1}:M_{1})M_{1}}}=(\bar{0})$ and
$\overline{(\sqrt{(K_{2}:M_{2})M_{2}}:M_{2})}$ $ \overline{\sqrt{(
N_{2}:M_{2})M_{2}}}=(\bar{0})$. Since $\bar{M_{1}}$ and
$\bar{M_{2}}$ are prime modules,
$(\sqrt{(K_{1}:M_{1})M_{1}}:M_{1})=(Q_{1}:M_{1})$ or
$\sqrt{(N_{1}:M_{1})M_{1}}= Q_{1}$ and
$(\sqrt{(K_{2}:M_{2})M_{2}}:M_{2})=(Q_{2}:M_{2})$ or
$\sqrt{(N_{2}:M_{2})M_{2}}= Q_{2}$. Therefore $G(\tau_T)$ is a
complete bipartite graph with two parts $U$ and $V$ such that
$N\in U$ if and only if $V(N)=V(M_{1}\oplus Q_{2})$ and $K\in V$
if and only if $V(K)=V(Q_{1}\oplus M_{2})$.
\end{proof}

\begin{cor}\label{c3.2} Let $\bar{M}$ be a faithful module and
does not have a non-zero submodule
 $\overline{\cap_{P\in T}P}\neq \bar{N}$ with $V(N)=T$.
Then the following statements are equivalent.

\begin {itemize}
\item [(a)] There is a vertex of $G(\tau_{Spec(M)})$ which is
adjacent to every other vertex of $G(\tau_{Spec(M)})$. \item [(b)]
$G(\tau_{Spec(M)})$ is a star graph. \item [(c)] $M=F\oplus D$,
where $F$ is a simple module and $D$ is a prime module.
\end{itemize}
\end{cor}

\begin{proof}

$(a) \Rightarrow (b)$ Let $\bar{M}$ be a faithful module. Then
$Q=(0)$ and we have $T=Spec(M)$. By Proposition \ref{p3.1},
$M=M_{1}\oplus M_{2}$, where $M_{1}$ is a simple module and
$M_{2}$ is a prime module. Then every non-zero submodule of $M$ is
of the form $M_{1}\oplus N_{2}$ and $(0)\oplus N_{2}$, where
$N_{2}$ is a non-zero submodule of $M_{2}$. By our hypothesis, we
can not have any vertex of the form $M_{1}\oplus N_{2}$, where
$N_{2}$ is a non-zero proper submodule of $M_{2}$. Also
$M_{1}\oplus (0)$ is adjacent to every other vertex, and non of
the submodules of the form $(0)\oplus N_{2}$ can be adjacent to
each other. So $G(\tau_{Spec(M)})$ is a star graph.

$(b) \Rightarrow (c)$ This follows by Proposition \ref{p3.1} (a).

$(c)\Rightarrow (a)$ Assume that $M=F\oplus D$, where $F$ is a
simple module and $D$ is a prime module. It is easy to see that
for some minimal submodule $N$ of $M$, we have $N^{2}\neq (0)$.
Since $M$ is a faithful module, Lemma \ref{l2.5} implies that
$F\cong eM$, where $e$ is an idempotent element of $R$. Finally
Proposition \ref{p3.1} (a) completes the proof.
\end{proof}

\begin{lem}\label{l3.3}  Let $e\in R$ be an idempotent element of $R$
 and $\bar{M}$ does not have a
non-zero submodule
 $\overline{\cap_{P\in T}P}\neq \bar{N}$ with $V(N)=T$. If
$G(\tau_T)$ is a triangle-free graph, then both $\bar{M_{1}}$ and
$\bar{M_{2}}$ are prime $R$-modules. Moreover, if $G(\tau_T)$ has
no cycle, then $\bar{M_{1}}$ is a simple module and $\bar{M_{2}}$
is a prime module.
\end{lem}
\begin{proof}
Without loss of generality, we can assume that $\bar{M_{1}}$ is a
prime module. Then $\bar{I}\bar{K}= (\bar{0})$, where
$(Q_{2}:M_{2})\subsetneq I<R$ and $Q_{2}\subsetneq K< M$. It
follows that $V(M_{1}\oplus K) \cup V(Q_{1}\oplus
IM_{2})=V(Q_{1}\oplus K(IM_{2}))=T$ (if $IM_{2}=K$, then
$V(Q_{1}\oplus K)=V(Q_{1}\oplus K^{2})=V(Q_{1}\oplus
K(IM_{2}))=T$, a contradiction). So both $\bar{M_{1}}$ and
$\bar{M_{2}}$ are prime $R$-modules. Now suppose that $G(\tau_T)$
has no cycle. If none of $\bar{M_{1}}$ and $\bar{M_{2}}$ is a
simple module, then we choose non-trivial submodules $N_{i}$ in
$M_{i}$ for some $i = 1, 2$. So $N_{1} \oplus Q_{2}$, $Q_{1}
\oplus N_{2}$, $M_{1} \oplus Q_{2}$, and $Q_{1} \oplus M_{2}$ form
a cycle, a contradiction.
\end{proof}

\begin{cor}\label{c3.9}
Assume that $M$ is a multiplication module or a primeful module
and $\bar{M}$ does not have a non-zero submodule
 $\overline{\cap_{P\in T}P}\neq \bar{N}$ with $V(N)=T$. Then $G(\tau_T)$ is a star graph if and
 only if $\bar{M_{1}}$ is a simple module and $\bar{M_{2}}$ is a prime
module for some idempotent $e\in R$.
\end{cor}

\begin{proof}
First we note that if $\bar{M}$ is a multiplication module, then
for any non-zero submodule $\bar{N}$ of $\bar{M}$, we have
$V(N)\neq T$. The necessity is clear by Proposition \ref{p3.1}
(a). For the converse, assume that $\bar{M}=\bar{M_{1}} \oplus
\bar{M_{2}}$, where $\bar{M_{1}}$ is a simple module and
$\bar{M_{2}}$ is a prime for some idempotent $e\in R$. Using the
Proposition \ref{p3.1} (b), $G(\tau_T)$ is a complete bipartite
graph with two parts $U$ and $V$ such that $N\in U$ if and only if
$V(N)=V(M_{1}\oplus Q_{2})$ and $K\in V$ if and only if
$V(K)=V(Q_{1}\oplus M_{2})$. We claim that $|U|=1$. Otherwise,
$V(M_{1} \oplus Q_{2})=V(N_{1} \oplus Q_{2})$, where $Q_{1}\neq
N_{1}< M_{1}$. It follows that $\sqrt{(N_{1}:M_{1})M_{1}}=M_{1}$,
a contradiction (note that if $M$ is a multiplication module or a
primeful module, then $\sqrt{(N:M)M}\neq M$, where $N< M$). So
$G(\tau_T)$ is a star graph.
\end{proof}

\begin{thm}\label{t3.4}
If $G(\tau_T)$ is a tree, then $G(\tau_T)$ is a star graph.
\end{thm}

\begin{proof}
Suppose that $G(\tau_T)$ is not a star graph. Then $G(\tau_T)$ has
at least four vertices. Obviously, there are two adjacent vertices
$L$ and $K$ of $G(\tau_T)$ such that $|N(L)\setminus \{K\}|\geq 1$
and $|N(K)\setminus \{L\}|\geq 1$. Let $N(L)\setminus \{K\} =
\{L_{i}\}_{i\in \Lambda}$ and $N(K)\setminus \{L\} =
\{K_{j}\}_{j\in \Gamma}$. Since $G(\tau_T)$ is a tree, we have
$N(L)\cap N(K)=\emptyset$. By \cite[Theorem 3.4]{ah14},
$diam(G(\tau_T))\leq 3$. So every edge of $G(\tau_T)$ is of the
form $\{L, K\}$, $\{L, L_{i}\}$ or $\{K, K_{j}\}$, for some $i\in
\Lambda$ and $j\in \Gamma$. Now, Pick $p\in \Lambda$ and $q\in
\Gamma$. Since $G(\tau_T)$ is a tree, $L_{p}K_{q}$ is a vertex of
$G(\tau_T)$. If $L_{p}K_{q}=L_{u}$ for some $u\in \Lambda$, then
$V(KL_{u})=T$, a contradiction. If $L_{p}K_{q}=K_{v}$, for some
$v\in \Gamma$, then $V(LK_{v})=T$, a contradiction. If
$L_{p}K_{q}=L$ or $L_{p}K_{q}=K$, then $V(L^{2}) =T$ or $V(K^{2})
=T$, respectively and hence $V(L) =T$ or $V(K) =T$, a
contradiction. So the claim is proved.
\end{proof}
\begin{thm}\label{t3.5}
Let $R$ be an Artinian ring and let $M$ be a multiplication or a
primeful module. If $G(\tau_T)$ is a bipartite graph, then $|T|=2$
and $G(\tau_T)\cong K_{2}$.
\end{thm}

\begin{proof}
First we may assume that $G(\tau_T)$ is not empty. Then $R$ can
not be a local ring. Otherwise, $T=V(mM)$, where $m$ is the unique
maximal ideal of $R$. Therefore \cite[Remark 2.6]{ah14} implies
that $mM=M$ and hence $T$ is empty, a contradiction. Hence by
\cite[Theorem 8.9]{ati69}, $R = R_{1}\oplus \ldots \oplus R_{n}$,
where $R_{i}$ is an Artinian local ring for $i=1, \ldots , n$ and
$n\geq 2$. By Lemma \ref{l2.2} and Proposition \ref{p2.3}, since
$G(\tau_T)$ is a bipartite graph, we have $n=2$ and hence
$\bar{M}\cong \bar{M_{1}}\oplus \bar{M_{2}}$ for some idempotent
$e\in R$. If $\bar{M_{1}}$ is a prime module, then it is easy to
see that $\bar{M_{1}}$ is a vector space over $R/Ann(\bar{M_{1}})$
and so is a semisimple $R$-module. A Similar argument as we did in
proof of Corollary \ref{c3.9} implies that $|T|=2$ and
$G(\tau_T)\cong K_{2}$.
\end{proof}

\begin{prop}\label{p3.6} Assume that $M$ is a multiplication module and
$Ann(\bar{M})$ is a nil ideal of $R$.

\begin{itemize}
\item [(a)] If $G(\tau_T)$ is a finite bipartite graph, then
$|T|=2$ and $G(\tau_T)\cong K_{2}$. \item [(b)] If $G(\tau_T)$ is
a regular graph of finite degree, then $|T|=2$ and $G(\tau_T)\cong
K_{2}$.
\end{itemize}
\end{prop}

\begin{proof}
$(a)$ By Theorem \ref{t2.10}, $\bar{M}$ is an Artinian and
Noetherian module so that $R/Ann(\bar{M})$ is an Artinian ring. A
similar arguments in Theorem \ref{t3.5} says that,
$R/Ann(\bar{M})$ is a non-local ring. So by \cite[Theorem
8.9]{ati69} and Lemma \ref{l2.2}, there exist pairwise orthogonal
idempotents modulo $Ann(\bar{M})$. By lemma \ref{l2.4},
$\bar{M}\cong \bar{M_{1}}\oplus \bar{M_{2}}$, for some idempotent
$e$ of $R$. Now, the proof that $G(\tau_T)\cong K_{2}$ is similar
to the proof of Corollary \ref{c3.9}.

$(b)$ We may assume that $G(\tau_T)$ is not empty. So $\bar{M}$ is
not a prime module by \cite[Remark 2.6]{ah14} and a similar manner
in proof of Theorem \ref{t2.10}, shows that $\bar{M}$ has a finite
length so that $R/Ann(\bar{M})$ is an Artinian ring. As in the
proof of part (a), $\bar{M}\cong \bar{M_{1}}\oplus \bar{M_{2}}$
for some idempotent $e\in R$. If $\bar{M_{1}}$ has one non-trivial
submodule $N$, then $deg(Q_{1} \oplus M_{2}) > deg(N \oplus
M_{2})$ (we note that by \cite[Proposition 2.5]{ah162}, $\bar{N}
\bar{K}=(\bar{0})$ for some $(\bar{0})\neq \bar{K}< \bar{M_{1}}$)
and this contradicts the regularity of $G(\tau_T)$. Hence
$\bar{M_{1}}$ is a simple module. Finally a similar argument as we
have seen in Corollary \ref{c3.9} gives $G(\tau_T)\cong K_{2}$.
\end{proof}

\begin{thm}\label{t3.7}
Assume that $\bar{M}$ does not have a non-zero submodule
 $\overline{\cap_{P\in T}P}\neq \bar{N}$ with $V(N)=T$, $Ann(\bar{M})$ is a nil ideal,
  and $|Min(\bar{M})|\geq 3$.
 Then $G(\tau_T)$ contains a cycle.
\end{thm}

\begin{proof}
If $G(\tau_T)$ is a tree, then by Theorem \ref{t3.4}, $G(\tau_T)$
is a star graph. Suppose that $G(\tau_T)$ is a star graph and $N$
is the center of star. Clearly, one can assume that
$\overline{\sqrt{(N:M)M}}$ is a minimal submodule of $\bar{M}$. If
$(\overline{\sqrt{(N:M)M}})^{2}\neq (\bar{0})$, then by Lemma
\ref{l2.4}, there exists an idempotent $e\in R$ such that
$(\overline{\sqrt{(N:M)M}})=\overline{eM}$. Now by Proposition
\ref{p2.3} and Lemma \ref{l3.3}, we conclude that
$|Min(\bar{M})|=2$, a contradiction. Hence
$(\overline{\sqrt{(N:M)M}})^{2}=(\bar{0})$ and hence $V(N)=T$, a
contradiction. Therefore $G(\tau_T)$ contains a cycle.
\end{proof}

\section{Coloring of the Zariski-topology graph of modules}
The purpose of this section is to study of coloring of the Zariski
topology-graph of modules and investigate the interplay between
$\chi(G(\tau_T))$ and $\omega(G(\tau_T))$. We note that since
$E(G(\tau_T))\geq 1$ when
 $G(\tau_T)\neq \emptyset$, then $\chi(G(\tau_T)))\geq 2$.

\begin{thm}\label{t4.1}
Let $\bar{M}$ be an Artinian module such that for every minimal
submodule $\bar{N}$ of $\bar{M}$, $N$ is a vertex in $G(\tau_T)$.
Then $\omega(G(\tau_T))= \chi(G(\tau_T))$.
\end{thm}
\begin{proof}
$\bar{M}$ is Artinian, so it contains a minimal submodule. Since
for every minimal submodule $\bar{N}$ of $\bar{M}$, $N$ is a
vertex in $G(\tau_T)$, we have $V(N)\neq T$. Also, $N\cap L=Q$,
where $\bar{N}$ and $\bar{L}$ are minimal submodules of $\bar{M}$.
It follows that $N$ and $L$ are adjacent in $G(\tau_T)$, where
$\bar{N}$ and $\bar{L}$ are minimal submodules of $\bar{M}$.
First, suppose that $\bar{M}$ has infinitely many minimal
submodules. Then $\omega(G(\tau_T))=\infty$ and there is nothing
to prove. Next, assume that $\bar{M}$ has $k$ minimal submodules,
where $k$ is finite. We conclude that $\chi(G(\tau_T))=
k=\omega(G(\tau_T))$. Obviously, $\omega(G(\tau_T))\geq k$. If
possible, assume that $\omega(G(\tau_T))> k$. Let
$\Sigma=\{N_{\lambda}\}_{\lambda\in I}$¸I, where
$|I|=\omega(G(\tau_T))$ be a maximum clique in $G(\tau_T)$. As
every $N_{\lambda}\in \omega$, $\overline{
\sqrt{(N_{\lambda}:M)M}}$ contains a minimal submodule, there
exists a minimal submodule $\bar{K}$ and submodules $N_{i}$ and
$N_{j}$ in $\omega$, such that $\bar{K}\subset \overline
{\sqrt{(N_{i}:M)M}}$ $\cap \overline{ \sqrt{(N_{j}:M)M}}$, and
hence $V(K)=T$, a contradiction. Hence $\omega(G(\tau_T))=k$.
Next, we claim that $G(\tau_T)$ is $k$-colorable. In order to
prove, put $A=\{\bar{K_{1}}, \ldots, \bar{K_{k}}\}$ be the set of
all minimal submodules of $\bar{M}$. Now, we define a coloring $f$
on $G(\tau_T)$ by setting $f(N) = min\{i|$ $K_{i}\subseteq
\sqrt{(N:M)M}\}$ for every vertex $N$ of $G(\tau_T)$. Let $N$ and
$L$ be adjacent in $G(\tau_T)$ and $f(N) = f(L) = j$. Thus
$K_{j}\subseteq \sqrt{(N:M)M}\cap \sqrt{(L:M)M}$, a contradiction.
It implies that $f$ is a proper $k$ coloring of $G(\tau_T)$ and
hence $\chi(G(\tau_T))\leq k=\omega(G(\tau_T))$, as desired.
\end{proof}

\begin{thm}\label{t4.2} Assume that $\bar{M}$ is a faithful module.
Then the following statements are equivalent.

\begin{itemize}
\item [(a)] $\chi(G(\tau_{Spec(M)}))=2$. \item [(b)]
$G(\tau_{Spec(M)})$ is a bipartite graph with two non-empty parts.
\item [(c)] $G(\tau_{Spec(M)})$ is a complete bipartite graph with
two non-empty parts. \item [(d)] Either $R$ is a reduced ring with
exactly two minimal prime ideals or $G(\tau_{Spec(M)})$ is a star
graph with more than one vertex.
\end{itemize}

\end{thm}
\begin{proof}
By using Lemma \ref{l2.8}, $G(\tau_{Spec(M)})$ and $AG(M)^{*}$ are
the same and so \cite[Theorem 3.2]{ah161} completes the proof.
\end{proof}

\begin{lem}\label{l4.3} Assume that $T$ is a finite set.
Then $\chi(G(\tau_T)))$ is finite. In particular,
$\omega(G(\tau_T)))$ is finite.
\end{lem}

\begin{proof}
Suppose that $T=\{P_{1}, P_{2},\ldots, P_{k}\}$ is a finite set of
distinct prime submodules of $M$. Define a coloring
$f(N)=min\{n\in \Bbb N|$ $P_{n}\notin V(N) \}$, where $N$ is a
vertex of $G(\tau_T)$. We can see that $\chi(G(\tau_T)))\leq k$.
\end{proof}

\begin{thm}\label{t4.4}
For every module $M$,  $\omega(G(\tau_T)) = 2$ if and only if
$\chi(G(\tau_T))=2$. In particular, $G(\tau_T)$ is bipartite if
and only if $G(\tau_T)$ is triangle-free.
\end{thm}

\begin{proof}
Let $\omega(G(\tau_T)) = 2$. On the contrary assume that
$G(\tau_T)$ is not bipartite. So $G(\tau_T)$ contains an odd
cycle. Suppose that $C:= N_{1} - N_{2} - \ldots -  N_{2k+1} -
N_{1}$ be a shortest odd cycle in $G(\tau_T)$ for some natural
number $k$. Clearly, $k\geq 2$. Since $C$ is a shortest odd cycle
in $G(\tau_T)$, $N_{3}N_{2k+1}$ is a vertex. Now consider the
vertices $N_{1}, N_{2}$, and $N_{3}N_{2k+1}$. If
$N_{1}=N_{3}N_{2k+1}$, then $V(N_{4} N_{1})= T$. This implies that
$ N_{1} - N_{4} - \ldots - N_{2k+1} - N_{1}$ is an odd cycle, a
contradiction. Thus $N_{1}\neq N_{3}N_{2k+1}$. If $N_{2}=
N_{3}N_{2k+1}$, then we have $C_{3}= N_{2} - N_{3}$ - $N_{4} -
N_{2}$, again a contradiction. Hence $N_{2}\neq N_{3}N_{2k+1}$. It
is easy to check $N_{1}, N_{2}$, and $N_{3}N_{2k+1}$ form a
triangle in $G(\tau_T)$, a contradiction. The converse is clear.
In particular, we note that empty graphs are bipartite graphs.
\end{proof}

\begin{cor}
Assume that $e\in R$ is an idempotent element and $\bar{M}$ does
not have a non-zero submodule
 $\overline{\cap_{P\in T}P}\neq \bar{N}$ with $V(N)=T$. Then $G(\tau_T)$ is
a complete bipartite graph if and only if $\bar{M_{1}}$ and
$\bar{M_{2}}$ are prime modules.
\end{cor}
\begin{proof}
Assume that $G(\tau_T)$ is a complete bipartite graph. Therefore
Theorem \ref{t4.4} states that $G(\tau_T)$ is a triangle-free
graph. So Lemma \ref{l3.3} follows that $\bar{M_{1}}$ and
$\bar{M_{2}}$ are prime modules. The conversely holds by
Proposition \ref{p3.1} (b).
\end{proof}

\begin{rem}\label{r4.5}
Assume that $S$ is a multiplicatively closed subset of $R$ such
that $S\cap (\cup_{P\in T}(P:M))=\emptyset$. Let $T_{S}=\{S^{-1}P:
P\in T\}$. One can see that $V(N)=T$ if and only if
$V(S^{-1}N)=T_{S}$, where $M$ is a finitely generated module.
\end{rem}

\begin{thm}\label{t4.6}
Let $S$ be a multiplicatively closed subset of $R$ defined in
Remark \ref{r4.5} and $M$ is a finitely generated module. Then
$G(\tau_{T_{S}})$ is a retract of $G(\tau_T)$ and
$\omega(G(\tau_{T_{S}}))= \omega(G(\tau_T))$.
\end{thm}
\begin{proof}
Consider a vertex map $\phi: V(G(\tau_T)) \longrightarrow
V(G(\tau_{T_{S}})), N\longrightarrow N_{S}$. Clearly, $N_{S}\neq
K_{S}$ implies that $N\neq K$ and $V(N)\cup V(K)=T$ if and only if
$V(N_{S})\cup V(K_{S})= T_{S}$. Thus $\phi$ is surjective and
hence $\omega(G(\tau_{T_{S}}))\leq \omega(G(\tau_T))$. If $N \neq
K$ and $V(N)\cup V(K)=T$, then we show that $N_{S}\neq K_{S}$. On
the contrary suppose that $N_{S}= K_{S}$. Then $V(N_{S}^{2})
=V(N_{S}K_{S})=V(N_{S})\cup V(K_{S})= T_{S}$ and so $V(N^{2})=T$,
a contradiction. This shows that the map $\phi$ is a graph
homomorphism. Now, for any vertex $N_{S}$ of $G(\tau_{T_{S}})$, we
can choice a fixed vertex $N$ of $G(\tau_T)$. Then $\phi$ is a
retract (graph) homomorphism which clearly implies that
$\omega(G(\tau_{T_{S}}))= \omega(G(\tau_T))$ under the assumption.
\end{proof}

\begin{cor}\label{c4.7}
Let $S$ be a multiplicatively closed subset of $R$ defined in
Remark \ref{r4.5} and let $M$ be a finitely generated module. Then
$\chi(AG(M_{S}))= \chi(AG(M))$.
\end{cor}

\begin{cor}\label{c4.10}
Assume that $M$ is a semiprime module and $AG(M)^{*}$ does not
have an infinite clique. Then $M$ is a faithful module and
$0=(P_{1}\cap \ldots \cap P_{k}:M)$, where $P_{i}$ is a prime
submodule of $M$ for $i=1, \ldots , k$.
\end{cor}
\begin{proof}
By \cite[Theorem 3.7 (b)]{ah161}, $M$ is a faithful module and the
last assertion follows directly from the proof of \cite[Theorem
3.7 (b)]{ah161}.
\end{proof}

\begin{prop}\label{p4.8}
Let $\bar{M}$ be a cyclic module and let $T$ be a closed subset of
$Spec(M)$. We have the following statements.
\begin{itemize}
\item[(a)] If $\{P_{1}, \ldots, P_{n}\}\subseteq Min(T)$, then
there exists a clique of size $n$ in $G(\tau_T)$.\item[(b)] We
have $\omega(G(\tau_T))\geq |Min(T)|$ and if $|Min(T)|\geq 3$,
then $gr(G(\tau_T))=3$.\item[(c)] If $\sqrt{(\bar{0})}=
(\bar{0})$, then $\chi(G(\tau_{Spec(M)})) =
\omega(G(\tau_{Spec(M)})) = |Min(T)|$.
\end{itemize}

\end{prop}
\begin{proof}
(a) The proof is straightforward by the facts that
$AG(\bar{M})=AG(\bar{M})^{*}$ has a clique of size $n$ by
\cite[Theorem 2.18]{ah162} and $AG(\bar{M})$ is isomorphic with a
subgraph of $G(\tau_T)$ by Lemma \ref{l2.7}.

(b) This is clear by item (a).

(c) If $|Min(T)|=\infty$, then by Proposition \ref{p4.8} (b),
there is nothing to prove. Otherwise, \cite[Theorem 2.20]{ah162}
implies that $AG(\bar{M})$ does not have an infinite clique. So
$\bar{M}$ is a faithful module by Corollary \ref{c4.10}. Next,
Lemma \ref{l2.8} says that $G(\tau_{Spec(M)})$ and $AG(M)^{*}$ are
the same. Now the result follows by \cite[Theorem 2.20]{ah162}.
\end{proof}

\begin{lem}\label{l4.12} Assume that $\bar{M}$ is a semiprime module. Then
the following statements are equivalent.

\begin{itemize} \item [(a)] $\chi(G(\tau_{Spec(M)})))$ is finite. \item [(b)]
$\omega(G(\tau_{Spec(M)})))$ is finite. \item [(c)]
$G(\tau_{Spec(M)}))$ does not have an infinite clique.
\end{itemize}
\end{lem}

\begin{proof}
$(a)\Longrightarrow (b)\Longrightarrow (c)$ is clear.

$(c)\Longrightarrow (d)$ Suppose that $G(\tau_{Spec(M)}))$ does
not have an infinite clique. By Lemma \ref{l2.7},
$AG(\bar{M})^{*}$ does not have an infinite clique and so by
Corollary \ref{c4.10}, there exists a finite number of prime
submodules $P_{1}, ... ,P_{k}$ of $M$ such that $(\cap_{P\in
T}P:M)=(P_{1}\cap \ldots \cap P_{k}:M)$. Define a coloring
$f(N)=min\{n\in \Bbb N|$ $P_{n}\notin V(N) \}$, where $N$ is a
vertex of $G(\tau_T)$. Then we have $\chi(G(\tau_{Spec(M)})))\leq
k$.
\end{proof}

\begin{cor}\label{c4.13}
Assume that $AG(M/\cap_{P\in T}P)^{*}$ does not have an infinite
clique. Then $G(\tau_{Spec(M)})$ and $AG(M)^{*}$ are the same.
Also, $\chi(G(\tau_{Spec(M)})))$ is finite.
\end{cor}
\begin{proof}
Since $M/\cap_{P\in T}P$ is a semiprime module, by Corollary
\ref{c4.10}, $M/\cap_{P\in T}P$ is a faithful module and there
exists a finite number of prime submodules $P_{1}, ... ,P_{k}$ of
$M$ such that $(\cap_{P\in T}P:M)=(P_{1}\cap \ldots \cap
P_{k}:M)$. So the result follows by Lemma \ref{l2.8} and from the
proof of $(c)\Longrightarrow (d)$ of Lemma \ref{l4.12}.
\end{proof}

We recall that $M$ is said to be \textit{X-injective} if either
$X=\emptyset$ or the natural map of $X=Spec(M)$ is injective (see
\cite{HR101}).

\begin{prop}\label{p4.14} Suppose that $\sqrt{(\bar{0})}=
(\bar{0})$, for every
 minimal member $P$ of $T$, $(P:M)$
 is a minimal ideal of $R$, and
 $\bar{M}$ is an $X$-injective module. Then
the following statements are equivalent.

\begin{itemize} \item [(a)] $\chi(G(\tau_{Spec(M)}))$ is finite. \item [(b)]
$\omega(G(\tau_{Spec(M)}))$ is finite. \item [(c)]
$G(\tau_{Spec(M)})$ does not have an infinite clique. \item [(d)]
$Min(T)$ is a finite set.
\end{itemize}
\end{prop}

\begin{proof}
$(a)\Longrightarrow (b)\Longrightarrow (c)$ is clear.

$(c)\Longrightarrow (d)$ Suppose $G(\tau_{Spec(M)})$ does not have
an infinite clique. By Lemma \ref{l2.7}, $AG(\bar{M})^{*}$ does
not have an infinite clique and hence by Corollary \ref{c4.10},
there exists a finite number of prime submodules $P_{1}, ...
,P_{k}$ of $M$ such that $(\cap_{P\in T}P:M)=(P_{1}\cap P_{2}\cap
... \cap P_{k}:M)$. By assumptions, one can see that $Min(T)$ is a
finite set.

$(d)\Longrightarrow (a)$ Assume that $Min(T)$ is a finite set
(equivalently, $\bar{M}$ has a finite number of minimal prime
submodules) so that $(\cap_{P\in T}P:M)=(P_{1}\cap P_{2}\cap ...
\cap P_{k}:M)$, where $Min(T)=\{P_{1}, ... ,P_{k}\}$. Define a
coloring $f(N)=min\{n\in N|$ $P_{n}\notin V(N) \}$, where $N$ is a
vertex of $G(\tau_{Spec(M)})$. Then we have
$\chi(G(\tau_{Spec(M)}))\leq k$.
\end{proof}

\begin{eg}\label{e4.15}
If $M$ is a faithfully flat $R$-module (for example, free
modules), then $pM$ is a $p$-prime submodule of $M$, where $p$ is
a prime ideal of $R$ by \cite[Theorem 3]{lu84}. So for every
 minimal prime submodule $P$ of $M$, $(P:M)$
 is a minimal ideal of $R$.
\end{eg}

\begin{prop}\label{p4.16} Assume that $\sqrt{(\bar{0})}=
(\bar{0})$ and
 $\bar{M}$ is a faithful module.
Then the following statements are equivalent.

\begin{itemize} \item [(a)] $\chi(G(\tau_{Spec(M)}))$ is finite. \item [(b)]
$\omega(G(\tau_{Spec(M)}))$ is finite. \item [(c)]
$G(\tau_{Spec(M)})$ does not have an infinite clique. \item [(d)]
$R$ has a finite number of minimal prime ideals. \item [(e)]
$\chi(G(\tau_{Spec(M)}))=\omega(G(\tau_{Spec(M)}))=|Min(R)|=k$,
where $k$ is finite.
\end{itemize}
\end{prop}
\begin{proof}
This is clear by Lemma \ref{l2.8}, \cite[Proposition 3.11]{ah161},
and \cite[Corollary 3.12]{ah161}.
\end{proof}

\end{document}